\documentclass[11pt]{amsart}

\usepackage{amsmath}
\usepackage{amssymb}
\usepackage{bbm}
\usepackage{pdfsync}
\usepackage{esint} 
\usepackage[breaklinks=true]{hyperref}
\usepackage[utf8]{inputenc}
\usepackage[utf8]{inputenc}
\usepackage[T1]{fontenc}

\newcommand{\R}{{\mathbb R}}       

\newcommand{\DD}{{\mathcal D}}

\newcommand{\AZ}{{\mathcal A}}
\newcommand{\BZ}{{\mathcal B}}

\newcommand{\RR}{{\mathcal R}}

\newcommand{\diam}{{\rm diam}}
\newcommand{\dist}{{\rm dist}}

\newcommand{\rf}[1]{{(\ref{#1})}}

\newcommand{\supp}{\operatorname{supp}}

\newcommand{\ve}{{\varepsilon}}
\newcommand{\vv}{{\vspace{2mm}}}

\newcommand{\wt}[1]{{\widetilde{#1}}}

\newcommand{\pv}{\operatorname{pv}}

\newcommand{\BAUP}{{\mathsf{BAUP}}}
\newcommand{\BAUPP}{{\mathsf{BAUPP}}}

\def\XXint#1#2#3{{\setbox0=\hbox{$#1{#2#3}{\int}$ }
\vcenter{\hbox{$#2#3$ }}\kern-.58\wd0}}

\def\H11{\textup{H}^1_1} 					

\usepackage{pgf,tikz} 

\definecolor{ffffff}{rgb}{1.0,1.0,1.0}
\definecolor{qqqqff}{rgb}{0.0,0.0,1.0}
\definecolor{ffqqqq}{rgb}{1.0,0.0,0.0}
\definecolor{zzzzqq}{rgb}{0.6,0.6,0.0}
\definecolor{marronet}{rgb}{0.6,0.2,0}
\definecolor{negre}{rgb}{0,0,0}
\definecolor{vermell}{rgb}{0.8,0.05,0.05}
\definecolor{blau}{rgb}{0.3,0.2,1.}
\definecolor{blauclar}{rgb}{0.,0.,1.}
\definecolor{grisfosc}{rgb}{0.25098039215686274,0.25098039215686274,0.25098039215686274}
\definecolor{verd}{rgb}{0.1,0.6,0.1}
\definecolor{taronja}{rgb}{0.9,0.6,0.05}
\definecolor{vermellclar}{rgb}{1.,0.,0.}
\definecolor{verdet}{rgb}{0,0.8,0.1}
\definecolor{blauverd}{rgb}{0,0.4,0.2}
\definecolor{grisclar}{rgb}{0.6274509803921569,0.6274509803921569,0.6274509803921569}

\newtheorem{theorem}{Theorem}[section]
\newtheorem{lemma}[theorem]{Lemma}

\newtheorem*{claim*}{Claim}

\newtheorem*{theorem*}{Theorem}

\theoremstyle{definition}

\theoremstyle{remark}
\newtheorem{rem}[theorem]{\bf Remark}

\numberwithin{equation}{section}

\newcommand{\brem}{\begin{rem}}
\newcommand{\erem}{\end{rem}}

\begin{document}

\title{Riesz transforms and the BAUPP and BWGL criteria for uniform rectifiability}

\author{Xavier Tolsa}

\address{ICREA, Barcelona\\
Dept. de Matem\`atiques, Universitat Aut\`onoma de Barcelona \\
and Centre de Recerca Matem\`atica, Barcelona, Catalonia.}
\email{xavier.tolsa@uab.cat}

\thanks{Supported by the European Research Council (ERC) under the European Union's Horizon 2020 research and innovation programme (grant agreement 101018680). Also partially supported by MICIU (Spain) under the grant PID2024-160507NB-I00. 
}



\begin{abstract}
In this note it is shown that if $\mu$ is an $n$-Ahlfors regular measure in $\R^{n+1}$ such that
the $n$-dimensional Riesz transform is bounded in $L^2(\mu)$ and the so-called BAUPP (bilateral approximation by unions of parallel planes)
condition holds for $\mu$, then $\mu$ satisfies the BWGL (bilateral weak geometric lemma), and so $\mu$ is uniformly $n$-rectifiable.
In this way, one can solve the David-Semmes problem in codimension one without relying on the BAUP (bilateral approximation by unions of planes)
criterion of David and Semmes.
\end{abstract}

\newcommand{\mih}[1]{\marginpar{\color{red} \scriptsize \textbf{Mih:} #1}}
\newcommand{\xavi}[1]{\marginpar{\color{blue} \scriptsize \textbf{Xavi:} #1}}

\maketitle

\section{Introduction}

A Radon measure $\mu$ in $\R^{d}$ is called  $n$-Ahlfors regular if
\begin{equation}\label{eqAD1}
C^{-1}r^n\leq \mu(B(x,r))\leq C r^n \quad \mbox{ for all $x\in\supp\mu$ and $0<r\leq \diam(\supp\mu)$.}
\end{equation}
On the other hand, $\mu$ is called uniformly $n$-rectifiable if it is $n$-Ahlfors regular and
there exist constants $\theta, M >0$ such that for all $x \in \supp\mu$ and all $0<r\leq \diam(\supp\mu)$ 
there is a Lipschitz mapping $g$ from the ball $B_n(0,r)$ in $\R^{n}$ to $\R^d$ with $\text{Lip}(g) \leq M$ such that
$$
\mu(B(x,r)\cap g(B_{n}(0,r)))\geq \theta r^{n}.$$
The notion of uniform rectifiability is a quantitative version of $n$-rectifiability introduced by David and Semmes (see \cite{DS1} and
\cite{DS}) which has attracted much attention because its applications to harmonic analysis and PDE's in rough settings, among other things.

An important question in this area is the so-called David-Semmes problem, which consists in proving that if $\mu$ is $n$-Ahlfors regular and the
Riesz transform operator $\RR_\mu$ is bounded in $L^2(\mu)$, then $\mu$ is uniformly $n$-rectifiable. 
The ($n$-dimensional) Riesz transform of a signed Radon measure $\nu$ is defined by
$$\RR\nu (x) = \int \frac{x-y}{|x-y|^{n+1}}\,d\nu(y),$$
whenever the integral makes sense. 
We also write
$$\RR_*\nu(x) = \sup_{\ve>0} |\RR_{\ve}\nu(x)|, \qquad  \pv\RR\nu(x) = \lim_{\ve>0} \RR_{\ve}\nu(x),$$
in case that the latter limit exists. Remark that, sometimes, abusing notation we will
write $\RR\nu$ instead of $\pv\RR\nu$.

For $f\in L^1_{loc}(\mu)$ and a positive Radon measure $\mu$,
one writes $\RR_\mu f(x) = \RR(f\mu)(x)$.
Given $\ve>0$, the $\ve$-truncated Riesz transform of $\mu$ equals
$$\RR_\ve\mu (x) = \int_{|x-y|>\ve} \frac{x-y}{|x-y|^{n+1}}\,d\mu(y),$$
and the operator $\RR_{\mu,\ve}$ is defined by
 $\RR_{\mu,\ve}f(x) = \RR_\ve(f\mu)(x)$. 
 We say that $\RR_\mu$ is bounded in $L^2(\mu)$ if
the operators $\RR_{\mu,\ve}$ are bounded uniformly in $L^2(\mu)$ uniformly on $\ve$, and then we denote
$$\|\RR_\mu\|_{L^2(\mu)\to L^2(\mu)} = \sup_{\ve>0} \|\RR_{\mu,\ve}\|_{L^2(\mu)\to L^2(\mu)}.$$

The David-Semmes problem was solved in the affirmative in the case $n=1$ in \cite{MMV} by Mattila, Melnikov, and Verdera, and in the case
$n=d-1$ in \cite{NToV} by Nazarov, the author, and Volberg \cite{NToV}. For other values of $n$, this is still an open problem. One of the essential ingredients of the proof in \cite{NToV} is the 
so called BAUP criterion. To state this, we need need some additional notation.

For a given $\ve>0$ and $x\in\supp\mu$, $r>0$, we write $(x,r) \in \BAUP(\ve)$ (BAUP stands for bilateral approximation by a union of planes) if there exists a family $\tilde I_{x,r}$ of
$n$-planes such that, setting
$$\tilde H_{x,r} = \bigcup_{L\in \tilde I_{x,r}} L,$$
it holds
$$\sup_{x\in \supp\mu \cap B(x,r)} \frac{\dist(x,\tilde H_{x,r})}{\ell(Q)} + \sup_{x\in \tilde H_{x,r} \cap B(x,r)} \frac{\dist(x,\supp\mu)}{r}\leq \ve.$$
Then we say that $\mu$ satisfies the BAUP condition, if for any $\ve>0$, the following Carleson type condition holds:
$$\int_B \int_0^{r(B)} \chi_{(x,t)\not \in \BAUP(\ve)}\,\frac{dt}t\,d\mu(x) \leq \mu(B),$$
for any ball $B$ centered at $\supp\mu$, where $r(B)$ denotes the radius of $B$.
A deep result of David and Semmes (see \cite[Chapter II.3]{DS}) asserts that if $\mu$ is $n$-Ahlfors regular and it satisfies the BAUP condition, then $\mu$ is
uniformly $n$-rectifiable. The proof of this result is difficult and lengthy. The purpose of this note is to give a more straightforward argument
that avoids the application of the BAUP criterion for the solution of the David-Semmes problem in codimension $1$ and instead relies on the BGWL (bilateral weak geometric lemma). The definition of the BGWL condition is analogous to the one of the BAUP condition, with the family $\tilde I_{x,r}$ formed in this case by a single $n$-plane.
Remark that the proof of the fact that the BGWL suffices for the uniform $n$-rectifiability of $\mu$ is considerably easier than the one
for the BAUP condition, specially in codimension one (see \cite[Chapter II.2]{DS}).

In \cite{NToV} it is shown that if $\mu$ is an $n$-Ahlfors regular measure in $\R^{n+1}$ such that $\RR_\mu$ is bounded in $L^2(\mu)$, then the BAUP condition holds for $\mu$ and so $\mu$ is uniformly $n$-rectifiable, by the aforementioned criterion of David and Semmes. In fact, a quick inspection of the arguments in \cite{NToV} shows that $\mu$ satisfies something (a priori) stronger than the BAUP condition, namely what we call the BAUPP (bilateral approximation by union of {\em parallel} planes. This is defined as the BAUP, but only allowing parallel $n$-planes in family
$\tilde I_{x,r}$. Since this notion is important for our purposes we state it more precisely. 

For a given $\ve>0$ and $x\in\supp\mu$, $r>0$, we write $(x,r) \in \BAUPP(\ve)$  if the exists a family $I_{x,r}$ of parallel
$n$-planes such that, setting
$$H_{x,r} = \bigcup_{L\in I_{x,r}} L,$$
it holds
$$\sup_{x\in \supp\mu \cap B(x,r)} \frac{\dist(x,H_{x,r})}{\ell(Q)} + \sup_{x\in H_{x,r} \cap B(x,r)} \frac{\dist(x,\supp\mu)}{r}\leq \ve.$$
Then we say that $\mu$ satisfies the BAUPP condition, if for any $\ve>0$, the following holds:
$$\int_B \int_0^{r(B)} \chi_{(x,t)\not \in \BAUPÇP(\ve)}\,\frac{dt}t\,d\mu(x) \leq \mu(B),$$
for any ball $B$ centered at $\supp\mu$.

Remark that, trivially, we have
\begin{equation}\label{eqtriv}
BWGL \;\Rightarrow\; BAUPP\;\Rightarrow\; BAUP.
\end{equation}
The implication $BAUP\;\Rightarrow\; BWGL$ also holds, but it is much more complicated. Hence, the three conditions in \rf{eqtriv} are equivalent to the uniform rectifiability of $\mu$.

In this note we will prove the following.

\begin{theorem}\label{teo1}
Let $\mu$ be an $n$-Ahlfors regular measure in $\R^{n+1}$ satisfying the BAUPP condition such that the $n$-dimensional Riesz transform $\RR_\mu$ is bounded in $L^2(\mu)$.
Then $\mu$ satisfies the BWGL.
\end{theorem}

By the discussion above, it is clear that, combining this result with the arguments in \cite{NToV}, one deduces that the $L^2(\mu)$ boundedness of the $n$-dimensional Riesz transform $\RR_\mu$ implies the BWGL for $\mu$, and so its uniform $n$-rectifiability. In this way, we can skip the
BAUP criterion for the solution of the David-Semmes problem.

\vv

\section{Proof of Theorem \ref{teo1}}

For an $n$-Ahlfors regular measure $\mu$ such as the one in the theorem, we let $\DD_\mu$ be an associated David-Semmes dyadic lattice, such as the one introduced in \cite[Chapter I.3]{DS}.
For $Q\in\DD_{\mu,k}$, we denote $\ell(Q)=2^{-k}$, and we assume that $\diam(Q)\leq \ell(Q)$. We let $B_Q$ be a ball concentric with $Q$ with radius $2\ell(Q) = 2^{-k+1}$.
For a given $\ve>0$, we write $Q\in \BAUPP(\ve)$ if the exists a family $I_Q$ of parallel
$n$-planes intersecting $B_Q$ such that, setting
$$H_Q = \bigcup_{L\in I_Q} L,$$
it holds
$$\sup_{x\in \supp\mu \cap B_Q} \frac{\dist(x,H_Q)}{\ell(Q)} + \sup_{x\in H_Q \cap B_Q} \frac{\dist(x,\supp\mu)}{\ell(Q)}\leq \ve.$$
It is easy to check that $\mu$ satisfies the BAUPP condition if and only if, for any $\ve>0$, the family $\DD_\mu\setminus \BAUPP(\ve)$ is a Carleson family.
That is, for any $R\in\DD_\mu$,
\begin{equation}\label{eqfinal***}
\sum_{Q\in \DD_\mu(R)\setminus \BAUPP(\ve)}\mu(Q) \leq C(\ve)\,\mu(R).
\end{equation}

Our goal is to show that the 
 assumptions of Theorem \ref{teo1} imply that \rf{eqfinal***} holds for any $R\in\DD_\mu$. To this end, for a given $\delta>0$, denote 
$$\BZ(\delta)= \{Q\in\DD_\mu:b\beta_\infty(Q)>\delta\},$$
where
$$b\beta_\infty(Q) = \inf_L \bigg(\sup_{x\in \supp\mu \cap B_Q} \frac{\dist(x,L)}{\ell(Q)} + \sup_{x\in L \cap B_Q} \frac{\dist(x,\supp\mu)}{\ell(Q)}\bigg),$$
with the infimum taken over all families of $n$-planes $L$.
We need to show that $\BZ(\delta)$ is a Carleson family for any $\delta>0$, or equivalently, that $\mu$ satisfies the BWGL.

For any fixed $m\geq 0$, let $Q^{(m)}\in\DD_\mu$ be the $m$-th ancestor of $Q$. It is easy to check that the BAUPP condition for $\mu$ implies that for any $m>0$ and any $\ve>0$, the family $\AZ^m(\ve)$ of cubes
$Q\in\DD_\mu$ such that, for some  $0\leq j\leq m$,  $Q^{(j)}\not\in \BAUPP(\ve)$, is a Carleson family.
Therefore, it suffices to prove that the family
$$\BZ_m^*(\delta)= \BZ(\delta)\cap \{Q\in\DD_\mu:Q^{(j)}\in \BAUPP(\ve)\text{ for $1\leq j\leq m$}\}$$
is Carleson, taking $m$ large enough and $\ve$ small enough (in particular $\ve\leq \delta/100$, say), depending on $\delta$.
That is, it is enough to prove that, for any $R\in\DD_\mu$,
\begin{equation}\label{eqfinal}
\sum_{Q\in \DD_\mu(R)\cap \BZ_m^*(\delta)}\mu(Q) \leq C(\delta)\,\mu(R).
\end{equation}

As shown in \cite[Chapter II.2]{DS}, there exists some constant $C_1$ such that for any $Q\in\BAUPP(\ve)$ (with $\ve$ small enough), the family of $n$-planes $I_Q$ can be chosen so that it contains no more than $C_1$ $n$-planes (otherwise the $n$-Ahlfors regularity of $\mu$ would be violated).
We will choose $m$ below depending only on $C_1$ and the $n$-Ahlfors regularity of $\mu$.

Fix $R\in\DD_\mu$ and let $Q\in \BZ_m^*(\delta)$ with $Q\subset R$. By taking $\ve$ small enough, depending on $m$, by geometric arguments (modifying if necessary the $\ve$ parameter and the choice of the families $I_{Q^{(j)}}$) we can assume that
$$I_{Q} \subset I_{Q^{(1)}}\subset \ldots \subset I_{Q^{(m)}},$$
or equivalently
$$H_{Q} \subset H_{Q^{(1)}}\subset \ldots \subset H_{Q^{(m)}}.$$
Consequently, for any $N>2$ and $m>C_1 (N+1)$, it follows that there exists some $j\in [0,m-N]$ such that the sequence of cubes $Q^{(j)},\ldots,Q^{(j+N)} $ satisfies 
$$I_{Q^{(j)}} = I_{Q^{(j+1)}} = \ldots = I_{Q^{(j+N)}}.$$

Without loss of generality, suppose that the $n$-planes from $I_{Q^{(j)}}$ are horizontal, and let $L_t,L_b$ be the $n$-planes from $I_{Q^{(j)}}$ which
are in the top and in the bottom positions, respectively, so that all the other $n$-planes from $I_{Q^{(j)}}$ are located between $L_t$ and $L_b$.
We can assume that 
$$\dist(L_t,L_b)\geq \frac\delta4\,\ell(Q),$$
because otherwise $\mu$ is well $\delta$-approximated in $B_Q$ by a single $n$-plane such as $L_t$ (using also that $\ve\ll\delta$),  that is, $b\beta_\infty(Q)\leq \delta$, which
contradicts the fact that $Q\in\BZ_m^*(\delta)\subset\BZ(\delta)$.
Also, by translating  slightly $L_t$ and $L_b$ if necessary (at a distance at most $\ve\ell(Q^{(j)})$ from the respective original planes $L_t$ and $L_b$), we can assume that $\supp\mu\cap B_{Q^{(j+N)}}$ lies between $L_t$ and $L_b$.

Let $x_t\in L_t\cap B_{Q^{j}}$ and $x_b\in L_b\cap B_{Q^{j}}$ be two points having the same vertical projection, and denote by $S_Q$ the region 
comprised between $L_t$ and $L_b$, That is,
$$S_Q = \{y\in\R^{n+1}:x_{b,n+1}\leq  y_{n+1}\leq  x_{t,n+1}\},$$
where the subindex $n+1$ denotes the $(n+1)$-th coordinate, so that $\supp\mu\cap B_{Q^{(j+N)}}\subset S_Q$.

Consider a closed cylinder $C_Q$ with axis equal to the vertical line $\ell_{x_t,x_b}$ through $x_t$ and $x_b$, whose bases are contained in $L_t$ and $L_b$, with radius
$r\in[\ell(Q^{(j)})/2,\ell(Q^{(j)})]$. Remark that $C_Q\subset 2B_{Q^{(j)}}$.
Denote by $\partial_lC_Q$ the lateral boundary of $C_Q$, i.e., $\partial_lC_Q=\{y\in\R^{n+1}:
\dist(y,\ell_{x_t,x_b})=r\}$. We choose $r$ so that the following thin boundary condition holds:
\begin{equation}\label{eqthin}
\mu(\{x\in 3B_{Q^{(j)}}:\dist(x,\partial_lC_Q)\leq t\,r\}) \leq C\,t\,r^n \approx t\,\mu(C_Q)\quad \text{ for $0<t\leq1$.}
\end{equation}
The existence of such radius $r$ follows by standard methods.

Consider an $n$-plane $L_t'$ parallel to $L_t$ contained in $S_Q$ such that
$$3\ve\ell(Q^{(j+N)})\leq \dist(L_t,L_t') \leq 6\ve\ell(Q^{(j+N)})$$
whch is at a distance at least $c\ve \ell(Q^{(j+N)})$ from the other $n$-planes from $I_{Q^{(j)}}$.
Denote by $P_t$ the closed cylinder with axis $\ell_{x_t,x_b}$, with radius $r$, and bases contained in in $L_t$ and $L_t'$, so that $P_t\subset C_Q$ and the top bases of $C_Q$ and $P_t$ coincide.
Analogously,  let $L_b'$ be an $n$-plane parallel to $L_b$ contained in $S_Q$ such that
$$3\ve\ell(Q^{(j+N)})\leq \dist(L_b,L_b') \leq 6\ve\ell(Q^{(j+N)})$$
which is at a distance at least $c\ve \ell(Q^{(j+N)})$ from the other $n$-planes from $I_{Q^{(j)}}$, and
let $P_b$ be the closed cylinder with axis $\ell_{x_t,x_b}$, with radius $r$, and bases contained in $L_b$ and $L_b'$, so that $P_b\subset C_Q$ and the bottom bases of $C_Q$ and $P_b$ coincide. Remark that the $\BAUPP(\ve)$ condition of $Q^{(j)}$ ensures that 
$$\mu(P_t)\approx \mu(P_b) \approx \ell(Q^{(j)})^n\approx r^n.$$

\vv

\begin{lemma}\label{lem1}
Under the above assumptions, for $Q\in\DD_\mu \cap \BZ_m^*(\delta)$, assuming $\ve>0$ small enough and $N$ large enough, we have
$$|m_{P_t}(\RR_\mu(\chi_{2 R})) - m_{P_b}(\RR_\mu(\chi_{2 R}))|\geq \tau,$$
for some $\tau>0$ depending on the $n$-Ahlfors regularity of $\mu$.
\end{lemma}
\vv

In the lemma $m_{P}(f)$ stands for the $\mu$-mean of $f$ on $P$.
We defer the proof to the end of this note. 

By standard arguments
 from harmonic analysis\footnote{We provide some detailed arguments in Lemma \ref{lemaux} below.}, denoting by $P_t(Q)$ and $P_b(Q)$ the cylinders $P_t$ and $P_b$ associated with $Q$, we get
\begin{align}\label{eqfuf10}
\sum_{Q\in \DD_\mu(R)\cap \BZ_m^*(\delta)}\mu(Q) & \leq  \frac C{\tau^2}\sum_{Q\in \DD_\mu\cap \BZ_m^*(\delta)}
|m_{P_t(Q)}(\RR_\mu(\chi_{2 R})) - m_{P_b(Q)}(\RR_\mu(\chi_{2 R}))|^2 \,\mu(Q)\\
& \leq C(m,\delta,\tau)\, \|\RR_\mu(\chi_{2 R})\|_{L^2(\mu)}^2\leq C(m,\delta,\tau)\,\mu(R),\nonumber
\end{align}
which proves \rf{eqfinal}.

\vv

\begin{proof}[Proof of Lemma \ref{lem1}]
Denote by $S_Q^t$ the closed region comprised between $L_t$ and $L_t'$ and by $S_Q^b$ the one between $L_b$ and $L_b'$. Also, let
$S_Q^i$ the intermediate open region between $L_t'$ and $L_b'$, so that
$$S_Q = S_Q^t\cup S_Q^i \cup S_Q^b.$$
We put
\begin{align}\label{eqfuf1}
|m_{P_t}(\RR_\mu\chi_{2 R}) - m_{P_b}(\RR_\mu \chi_{2 R})| & \geq 
|m_{P_t}(\RR_\mu\chi_{2B_{Q^{(j+N)}}}) - m_{P_b}(\RR_\mu\chi_{2B_{Q^{(j+N)}}})| \\
&\quad - |m_{P_t}(\RR_\mu\chi_{2 R\setminus 2B_{Q^{(j+N)}}}) - m_{P_b}(\RR_\mu\chi_{2 R\setminus 2B_{Q^{(j+N)}}})|.\nonumber
\end{align}
 Since $P_t$ and $P_b$ are contained in $2B_{Q^{(j)}}$, by standard estimates we have that for all $x\in P_t$ and $y\in P_b$,
$$|\RR_\mu\chi_{2 R\setminus 2B_{Q^{(j+N)}}}(x) - \RR_\mu\chi_{2 R\setminus 2B_{Q^{(j+N)}}}(y)|\lesssim \frac{\ell(Q^{(j)})}{\ell(Q^{(j+N)})}
\approx 2^{-N}.$$
Thus,
\begin{equation}\label{eqfuf1.5}
|m_{P_t}(\RR_\mu\chi_{2 R\setminus 2B_{Q^{(j+N)}}}) - m_{P_b}(\RR_\mu\chi_{2 R\setminus 2B_{Q^{(j+N)}}})|\lesssim 2^{-N}.
\end{equation}

Next we will estimate from below the first term on the right hand side of \rf{eqfuf1}.
To shorten notation, from now on  we will write $B= B_{Q^{(j+N)}}$. We will also denote by $\RR^{n+1}$ be the vertical component of $\RR$. 
By the antisymmetry of $\RR^{n+1}$, we have
\begin{align*}
|m_{P_t}(\RR_\mu\chi_{2B}) - m_{P_b}&(\RR_\mu\chi_{2B})|\\ & \geq m_{P_t}(\RR_\mu^{n+1}\chi_{2B\setminus P_t}) - m_{P_b}(\RR_\mu^{n+1}\chi_{2B\setminus P_b})\\
& = m_{P_t}(\RR_\mu^{n+1}\chi_{2B\cap S_b}) + m_{P_t}(\RR_\mu^{n+1}\chi_{2B\cap S_i}) + m_{P_t}(\RR_\mu^{n+1}\chi_{2B\cap S_t\setminus P_t})\\
&\quad -m_{P_b}(\RR_\mu^{n+1}\chi_{2B\cap S_t})- m_{P_b}(\RR_\mu^{n+1}\chi_{2B\cap S_i}) - m_{P_b}(\RR_\mu^{n+1}\chi_{2B\cap S_b\setminus P_b}).
\end{align*}
Observe that $x_{n+1}-y_{n+1}>0$ for all $x\in P_t$ and $y\in S_i\cup S_b$. Thus
$$ m_{P_t}(\RR_\mu^{n+1}\chi_{2B\cap S_b})\geq m_{P_t}(\RR_\mu^{n+1}\chi_{P_b})\geq 0\quad \text{ and } m_{P_t}(\RR_\mu^{n+1}\chi_{2B\cap S_i})\geq 0.$$
Analogously,
$$m_{P_b}(\RR_\mu^{n+1}\chi_{2B\cap S_b})\leq m_{P_b}(\RR_\mu^{n+1}\chi_{P_t})\leq 0\quad \text{ and } m_{P_b}(\RR_\mu^{n+1}\chi_{2B\cap S_i})\leq 0.$$
Consequently,
\begin{align*}
|m_{P_t}(\RR_\mu\chi_{2B}) - m_{P_b}(\RR_\mu\chi_{2B})| 
& \geq m_{P_t}(\RR_\mu^{n+1}\chi_{P_b}) - m_{P_b}(\RR_\mu^{n+1}\chi_{P_t}) \\
&\quad - |m_{P_t}(\RR_\mu^{n+1}\chi_{2B\cap S_t\setminus P_t})|
-|m_{P_b}(\RR_\mu^{n+1}\chi_{2B\cap S_b\setminus P_b})|.
\end{align*}
From the $\BAUPP(\ve)$ condition and the separation between $P_t$ and $S_b$, it follows that for any $x\in P_t$ there exists a ball
$B_x\subset P_b$ centered in $\supp\mu\cap S_b$ with radius $r_x = \dist(L_t,L_b)/10\geq \delta\ell(Q)/40$ satisfying 
$$\dist(x,B_x)\approx \dist(L_t,L_b)\quad \text{ and }\quad
\mu(B_x\cap S_b)\approx r_x^n.$$
Then we deduce that
$$\RR_\mu^{n+1}\chi_{P_b}(x) \geq \RR_\mu^{n+1}\chi_{B_x}(x)\gtrsim \frac{\mu(B_x)}{\dist(x,B_x)^n}\approx 1.$$
Hence,
$$m_{P_t}(\RR_\mu^{n+1}\chi_{P_b})\gtrsim1,$$
and similarly 
$$-m_{P_b}(\RR_\mu^{n+1}\chi_{P_t})\gtrsim 1.$$
Therefore, there is some constant $\tau_0>0$ depending only on the $Ahlfors$-regularity of $\mu$ such that
\begin{equation}\label{eqfuf2}
|m_{P_t}(\RR_\mu\chi_{2B}) - m_{P_b}(\RR_\mu\chi_{2B})| \geq \tau_0- |m_{P_t}(\RR_\mu^{n+1}\chi_{2B\cap S_t\setminus P_t})|
-|m_{P_b}(\RR_\mu^{n+1}\chi_{2B\cap S_b\setminus P_b})|.
\end{equation}

Our next objective consists in showing that 
\begin{equation}\label{eqfuf4}
|m_{P_t}(\RR_\mu^{n+1}\chi_{2B\cap S_t\setminus P_t})| + |m_{P_b}(\RR_\mu^{n+1}\chi_{2B\cap S_b\setminus P_b})| \lesssim_N \ve^{1/2}.
\end{equation}
Observe that this estimate, together with \rf{eqfuf1}, \rf{eqfuf1.5}, and \rf{eqfuf2}, gives
$$|m_{P_t}(\RR_\mu\chi_{2 R}) - m_{P_b}(\RR_\mu \chi_{2 R})|  \geq \tau_0 - C2^{-N} - C(N)\ve^{1/2},$$
which proves the lemma if $N$ is large enough and $\ve$ small enough.
We will only estimate $|m_{P_t}(\RR_\mu^{n+1}\chi_{2B\cap S_t\setminus P_t})|$, since the arguments for $|m_{P_b}(\RR_\mu^{n+1}\chi_{2B\cap S_b\setminus P_b})|$ are similar. To this end, denote by $h$ the height of $P_t$, so that $h\approx \ve \ell(Q^{(j+N)})\approx 2^N\ve r$ (below we will take $\ve$ such that $2^N\ve\ll1$ and so the reader should thing that $P_t$ and $P_b$ are very flat cylinders).
Consider $x\in P_t$ and let $d(x):=\dist(x,\partial_lC_Q)$. Since the height of the region $S_b$ is $h$, we have $|x_{n+1}-y_{n+1}|\leq h$ for all $x,y\in S_b$, and so
$$|\RR_\mu^{n+1}\chi_{2B\cap S_t\setminus P_t})(x)|\leq \int_{y\in S_t\setminus P_t} \frac{\min(h,|x-y|)}{|x-y|^{n+1}}\,d\mu(y)
\leq \int_{|x-y|\geq d(x)} \frac{\min(h,|x-y|)}{|x-y|^{n+1}}\,d\mu(y)
.$$
It $d(x)> h$, then $|x-y|\geq h$ for $y$ in the domain of integration of the last integral. In this case, we write
$$|\RR_\mu^{n+1}\chi_{2B\cap S_t\setminus P_t})(x)|\leq \int_{|x-y|\geq d(x)} \frac{h}{|x-y|^{n+1}}\,d\mu(y)\lesssim \frac h{d(x)},$$
using the polynomial growth of $\mu$ for the last estimate.
In the case $d(x)\leq h$, we split
\begin{align*}
\int_{|x-y|\geq d(x)} \!\frac{\min(h,|x-y|)}{|x-y|^{n+1}}\,d\mu(y) &\leq \int_{d(x)\leq |x-y|\leq h}  \frac1{|x-y|^{n}}\,d\mu(y) +
 \int_{|x-y|> h}  \frac h{|x-y|^{n+1}}\,d\mu(y)\\
 & \lesssim \log\frac{2h}{d(x)} + 1\lesssim \bigg(\frac{h}{d(x)}\bigg)^{1/2},
\end{align*}
using again the polynomial growth of $\mu$.

Putting altogether, we get
\begin{align*}
\int_{P_t}|\RR_\mu^{n+1}\chi_{2B\cap S_t\setminus P_t})(x)|\,d\mu(x) & \leq \int_{x\in P_t:d(x)\leq h} \bigg(\frac{h}{d(x)}\bigg)^{1/2}\,d\mu(x) +
\int_{x\in P_t:d(x)> h}\frac h{d(x)}\,d\mu(x) \\ &= I_1 + I_2.
\end{align*}
We will estimate both $I_1$ and $I_2$ using the thin boundary condition \rf{eqthin}.
First we deal with $I_1$:
\begin{align*}
I_1 & = \sum_{k\geq0} \int_{x\in P_t:2^{-k-1}<d(x)\leq 2^{-k}h} \bigg(\frac{h}{d(x)}\bigg)^{1/2}\,d\mu(x)\\
&\lesssim \sum_{k\geq0} 2^{k/2}\,\mu(\{x\in P_t:d(x)\leq 2^{-k}h\} \lesssim \sum_{k\geq0} 2^{k/2}\,\frac{2^{-k} h}{r}\,\mu(P_t)
\lesssim \frac{ h}{r}\,\mu(P_t).
\end{align*}
Also,
\begin{align*}
I_2 & =  \sum_{k\geq0}
\int_{x\in P_t:2^kh< d(x)\leq 2^{k+1}h} \frac h{d(x)}\,d\mu(x) \lesssim \sum_{k\geq0}2^{-k}\,
\mu(\{x\in P_t:d(x)\leq 2^{k+1}h \}) \\
& \lesssim\sum_{\substack{k\geq0:\\2^kh\leq r}}
2^{-k}\,\frac{2^{k+1}h}r\,\mu(P_t) \approx \frac hr\,\log\frac rh\,\mu(P_t)\lesssim \bigg(\frac{h}r\bigg)^{1/2}\,\mu(P_t).
\end{align*}
Therefore,
$$\int_{P_t}|\RR_\mu^{n+1}\chi_{2B\cap S_t\setminus P_t})(x)|\,d\mu(x)\lesssim \bigg(\frac{h}r\bigg)^{1/2}\,\mu(P_t)\approx 2^{N/2}\ve^{1/2}\,\mu(P_t),$$
which, together with the analogous estimate for $P_b$,
 proves \rf{eqfuf4} and concludes the proof of the lemma.
\end{proof}

\vv
In the lemma below, we provide the detailed arguments for the estimate \rf{eqfuf10}.

\begin{lemma}\label{lemaux}
Under the above assumptions, we have
\begin{equation}\label{eqfuf7}
\sum_{Q\in \DD_\mu(R)\cap \BZ_m^*(\delta)}
|m_{P_t(Q)}(\RR_\mu(\chi_{2 R})) - m_{P_b(Q)}(\RR_\mu(\chi_{2 R}))|^2 \,\mu(Q)
 \leq C(m,\delta,\ve,\tau)\, \|\RR_\mu(\chi_{2 R})\|_{L^2(\mu)}^2.
\end{equation}
\end{lemma}

\begin{proof}
Our arguments below will use Lemma \ref{lem1} and the following standard result, whose proof can be found in  Section 14 from \cite{NToV}:

\vv
{\em Let $\mu$ be $n$-Ahlfors regular and let $\DD_\mu$ be the David-Semmes lattice associated with $\mu$. Let $\{\psi_Q\}_{Q\in\DD_\mu}$ be a family of Lipschitz functions satisfying the following, for fixed constants $C_2,C_3>1$:
\begin{itemize}
\item[(1)] $\supp\psi_Q\subset C_2B_Q$,
\item[(2)] $\int\psi_Q\,d\mu=0,$
\item[(3)] ${\rm Lip}(\psi_Q)\leq C_3\,\ell(Q)^{-\frac n2 -1}$.
\end{itemize}
Then, for any function $f\in L^2(\mu)$,
\begin{equation}\label{eqfuf8}
\sum_{Q\in\DD_\mu} \langle f,\psi_Q\rangle_\mu^2\lesssim \|f\|_{L^2(\mu)}^2,
\end{equation}
with the implicit constant depending only on $C_2$, $C_3$, and the $n$-Ahlfors regularity of $\mu$.
}
\vv

One could also prove this lemma  without using neither this result nor the estimate in Lemma \rf{lem1}, but our approach here suffices for our purposes.

Denote $f= \RR_\mu\chi_{2R}$. The sum on the left hand side of \rf{eqfuf7} can be written as follows
\begin{align*}
\sum_{Q\in \DD_\mu(R)\cap \BZ_m^*(\delta)} \left(\int \left(\frac{\chi_{P_t(Q)}}{\mu(P_t(Q))} - \frac{\chi_{P_b(Q)}}{\mu(P_b(Q))} \right)\,\mu(Q)^{1/2}f\,d\mu\right)^2 =\sum_{Q\in\DD_\mu(R)\cap \BZ_m^*(\delta)} \big\langle f,\eta_Q\big\rangle_\mu^2,
\end{align*}
where 
$$\eta_Q = \left(\frac{\chi_{P_t(Q)}}{\mu(P_t(Q))} - \frac{\chi_{P_b(Q)}}{\mu(P_b(Q))} \right)\,\mu(Q)^{1/2}.$$
We cannot apply \rf{eqfuf8} with $\eta_Q$ in place of $\psi_Q$ because $\eta_Q$ is not Lipschitz. To solve this issue, we will approximate
$\eta_Q$ by a a suitable Lipschitz function.

For each $Q\in\DD_\mu\cap \BZ_m^*(\delta)$, let $\wt\chi_{P_b(Q)}$ be a Lipschitz function which equals $1$ on $P_b(Q)$ and is supported in the $(t\ell(Q))$-neighborhood $U_{t\ell(Q)}(P_b)$, with ${\rm Lip}(\wt\chi_{P_b(Q)})\lesssim(t\,\ell(Q))^{-1}$, 
for a small $t$ to be chosen below, possibly depending on $\ve$ and $\delta$. We define $\wt \chi_{P_b}$ analogously, and then we let
$$\psi_Q = \left(\frac{\wt\chi_{P_t(Q)}}{\int\wt\chi_{P_t(Q)}\,d\mu}  - \frac{\wt\chi_{P_b(Q)}}{\int\wt\chi_{P_b(Q)}\,d\mu} \right)\,\mu(Q)^{1/2}.$$
It is immediate to check that the functions $\psi_Q$ satisfy the assumptions (1), (2), (3) above, possibly with constants depinding on $N$, $t$,
and other parameters.

Now, by the zero-mean value of $\eta_Q$ and $\psi_Q$,
\begin{align}\label{eqfuf9}
|\langle \eta_Q, f\rangle_\mu| & \leq |\langle \psi_Q, f\rangle_\mu| + |\langle (\eta_Q -\psi_Q), (f-m_Q(f))\rangle_\mu|\\
& \leq |\langle \psi_Q, f\rangle_\mu| + \|\eta_Q -\psi_Q\|_{L^2(\mu)} \left(\int_{2Q^{(j+N)}}|f-m_Q(f)|^2\,d\mu\right)^{1/2}\nonumber\\
&\leq |\langle \psi_Q, f\rangle_\mu| + C \|\eta_Q -\psi_Q\|_{L^2(\mu)} \mu(Q)^{1/2},\nonumber
\end{align}
where we used the boundedness of $\RR_\mu$ from $L^\infty(\mu)$ to $BMO(\mu)$ in the last inequality, and recalling that $f=\RR_\mu\chi_{2R}$.

Observe that $\wt\chi_{P_t(Q)}- \chi_{P_t(Q)}=\chi_{U_t(P_t(Q))\setminus P_t(Q)}$ and so
$$\|\wt\chi_{P_t(Q)}- \chi_{P_t(Q)}\|_{L^1(\mu)} + \|\wt\chi_{P_t(Q)}- \chi_{P_t(Q)}\|_{L^2(\mu)}^2 \lesssim \mu(U_t(P_t(Q))\setminus P_t(Q)) \lesssim_{N,\ve} t\,\mu(Q),$$
by the small boundary condition of $C_Q$ and the choice of the $n$-planes $L_t$, $L_t'$. The same estimate holds with $P_b$ instead of $P_t$.
Then it follows that
\begin{align*}
\|\eta_Q -\psi_Q\|_{L^2(\mu)}& \lesssim \frac{\|\wt\chi_{P_t(Q)}- \chi_{P_t(Q)}\|_{L^1(\mu)}}{\mu(Q)^2}\,\mu(P_t)^{1/2}\mu(Q)^{1/2}
+ \frac{\|\wt\chi_{P_t(Q)}- \chi_{P_t(Q)}\|_{L^2(\mu)}}{\mu(P_t)}\,\mu(Q)^{1/2}\\
&\quad\!\!\!\!\! +\frac{\|\wt\chi_{P_b(Q)}- \chi_{P_b(Q)}\|_{L^1(\mu)}}{\mu(Q)^2}\,\mu(P_b)^{1/2}\mu(Q)^{1/2}
+ \frac{\|\wt\chi_{P_b(Q)}- \chi_{P_b(Q)}\|_{L^2(\mu)}}{\mu(P_b)}\,\mu(Q)^{1/2}\\
&\lesssim_{N,\ve} \frac{t\,\mu(Q)}{\mu(Q)^2}\,\mu(Q) + \frac{t^{1/2}\,\mu(Q)^{1/2}}{\mu(Q)}\,\mu(Q)^{1/2}\lesssim_{N,\ve} t^{1/2},
\end{align*}
since $t<1$.
Plugging this estimate into \rf{eqfuf9}, we derive
$$|\langle f, \eta_Q\rangle_\mu| \leq |\langle f,\psi_Q\rangle_\mu| + C(\ve,N)\, t^{1/2}\mu(Q)^{1/2}.$$
By Lemma \ref{lem1}, we know that $|\langle f,\eta_Q\rangle_\mu|\geq \tau\,\mu(Q)^{1/2},$ where $\tau$ is an absolute constant. Therefore,
it we take $t$ small enough so that $C(\ve,N)\, t^{1/2}\leq \tau/2$, the last term on the right hand side in the display above can be absorbed by the term on the left hand side, giving
$$|\langle f,\eta_Q\rangle_\mu| \leq 2|\langle f,\psi_Q\rangle_\mu|.$$

Squaring, summing over $Q\in\DD_\mu(R)\cap\BZ^*_m(\delta)$, applying \rf{eqfuf8}, and using the $L^2(\mu)$ boundedness of the Riesz transform, we obtain
$$\sum_{Q\in \DD_\mu(R)\cap\BZ_m^*(\delta)} |\langle f,\eta_Q\rangle_\mu|^2 \leq 4\sum_{Q\in \DD_\mu(R)\cap\BZ_m^*(\delta)} |\langle f,\psi_Q\rangle_\mu|^2\lesssim
\|\RR_\mu\chi_{2R}\|_{L^2(\mu)}^2 \lesssim \mu(R),$$
concluding the proof of the lemma.
\end{proof}



\vv

\end{document}